\documentclass[11pt,a4paper,english,reqno]{amsart}
\usepackage{amsmath,amssymb,amsfonts,epsfig,mathrsfs}
\usepackage[T1]{fontenc}

\usepackage{color}
\usepackage{array}
\usepackage{amsthm}
\usepackage{amstext}
\usepackage{graphicx}
\usepackage{setspace}
\usepackage[margin=2.5cm]{geometry}
\usepackage{bbm}
\usepackage{color}
\usepackage{enumitem}
\allowdisplaybreaks[4]

\usepackage{pgfplots}

\usepackage{amscd,psfrag}
\usepackage{yhmath}
\usepackage[mathscr]{eucal}

\usepackage{slashed}

\makeatletter
\pdfpageheight\paperheight
\pdfpagewidth\paperwidth

\usepackage{mathrsfs}

\usepackage{epstopdf}
\usepackage{chngcntr}
\counterwithin{figure}{section}
\usepackage{mathrsfs}

\usepackage{indentfirst}	

\usepackage[normalem]{ulem}
\theoremstyle{plain}

\newtheorem{definition}{Definition}[section]
\newtheorem{theorem}[definition]{Theorem}
\newtheorem*{theorem*}{Theorem}

\newtheorem*{remark*}{Remark}
\newtheorem*{sideremark*}{Side Remark}\newtheorem*{mt*}{Main Theorem}

\newtheorem*{claim*}{Claim}
\newtheorem*{q*}{Question}
\newtheorem{lemma}[definition]{Lemma}

\newtheorem*{corollary*}{Corollary}
\newtheorem*{proposition*}{Proposition}

\newtheorem{proposition}[definition]{Proposition}

\newcommand{\R}{\mathbb{R}}

\newcommand{\na}{\nabla}

\newcommand{\dd}{{\rm d}}
\newcommand{\p}{\partial}
\newcommand{\e}{\epsilon}
\newcommand{\emb}{\hookrightarrow}

\newcommand{\map}{\rightarrow}

\newcommand{\1}{\mathbbm{1}}
\newcommand{\ball}{{\mathbb{B}}}
\newcommand{\bo}{{\mathcal{O}}}
\newcommand{\id}{{\bf id}}
\newcommand{\sna}{\slashed{\na}}
\newcommand{\stwo}{{\mathbb{S}^2}}

\newcommand{\xx}{\frac{x}{|x|}}
\newcommand{\T}{{\Theta}}

\newcommand{\D}{\mathcal{D}}
\newcommand{\ee}{{\mathscr{E}}}

\allowdisplaybreaks[4]

\def\XXint#1#2#3{{\setbox0=\hbox{$#1{#2#3}{\int}$ }
\vcenter{\hbox{$#2#3$ }}\kern-.6\wd0}}

\numberwithin{equation}{section}
\numberwithin{figure}{section}

\title{Stability of Minimising Harmonic Maps under $W^{1,p}$ Perturbations of Boundary Data: $p\geq 2$}

\author{Siran Li}
\address{Siran Li: New York University -- Shanghai, Office 1146, 1555 Century Avenue, Pudong District, Shanghai, China (200122)}

\address{NYU-ECNU Institute of Mathematical Sciences, Room 340, Geography Building, 3663 North Zhongshan Road, Shanghai, China (200062)}

\address{Current Address${}^\dagger$: School of Mathematical Sciences, Shanghai Jiao Tong University, No.~6 Science Buildings,
800 Dongchuan Road, Minhang District, Shanghai, China (200240)}

\email{\texttt{sl4025@nyu.edu}}


\date{\today}

\begin{document}

\maketitle

\begin{abstract}
Let $\Omega \subset \R^3$ be a Lipschitz domain. Consider a harmonic map $v: \Omega \map \stwo$ with boundary data $v|\p\Omega = \varphi$ which minimises the Dirichlet energy. For $p\geq 2$, we show that any energy minimiser $u$ whose boundary map $\psi$ has a small $W^{1,p}$-distance to $\varphi$ is close to $v$ in H\"{o}lder norm modulo bi-Lipschitz homeomorphisms, provided that $v$ is the unique minimiser attaining the boundary data. 
The index $p=2$ is sharp: the above stability result fails for $p<2$ due to the constructions by Almgren--Lieb \cite{al} and Mazowiecka--Strzelecki \cite{ms}.
\end{abstract}

\section{Introduction}

Let $u: \Omega \map \stwo$, where $\Omega$ is a Lipschitz domain in $\R^3$ and $\stwo$ is the unit $2$-sphere. We are concerned with the boundary value problem for the {\em harmonic map} equation:
\begin{equation}\label{harmonic map}
\begin{cases}
-\Delta u =|\na u|^2 u\qquad \text{ in }  \Omega,\\
u = \varphi\qquad \text{ on } \p\Omega.
\end{cases}
\end{equation}
This is the Euler--Lagrange equation for  minimisers of the Dirichlet energy
\begin{equation}
E[u] := \int_\Omega |\na u(x)|^2\,\dd x
\end{equation} 
over the space
\begin{equation}
W^{1,2}_\varphi(\Omega,\stwo) := \Big\{ u \in W^{1,2}(\Omega,\stwo) : u|\p\Omega = \varphi  \Big\}.
\end{equation}
The existence of minimisers are well-known for $\varphi \in W^{1/2,2}(\p\Omega,\stwo)$ in the sense of trace, due to the lower semi-continuity of the functional $E$. Also, for $\varphi \in W^{1,2}(\p\Omega,\stwo)$ the space $W^{1,2}_\varphi(\Omega,\stwo)$ is non-empty, as it contains the degree-$0$-homogeneous extension $\varphi(x\slash|x|)$. 

The weak solutions to \eqref{harmonic map} are called (weakly) {\em harmonic maps}. Minimisers of the Dirichlet integral clearly satisfy \eqref{harmonic map}, hence we call them {\em minimising harmonic maps}. The singular set of a harmonic map $u$, denoted by $sing\, u$, consists of the points that have an open neighbourhood in $\overline{\Omega}$ in which $u$ is not H\"{o}lder continuous --- equivalently, not real-analytic (\cite{su1, bg, moser}). We remark that there are non-minimising harmonic maps. As a prominent example,  Rivi\`{e}re \cite{r} constructed a harmonic map $v: \ball \map \stwo$ with $sing\, v = \overline{\ball}$, but Schoen--Uhlenbeck \cite{su1} proved that minimising harmonic map $u:\ball\map\stwo$ must have discrete singular sets ($\ball$ = the unit $3$-ball).

In this note, we study the stability of the minimising harmonic maps $u$  with respect to the boundary data $\varphi$. In a very interesting recent paper \cite{ms}, by elaborating on Almgren--Lieb's constructions in \cite{al}, Mazowiecka--Strzelecki proved that $u$ is highly non-stable under $W^{1,p}$-perturbations of $\varphi$ for $p<2$ and $\Omega=\ball$:
\begin{proposition}[Theorem 1.1 in \cite{ms}]\label{propn: p<2}
Let $\varphi \in C^\infty(\p\ball,\stwo)$ be a degree-$0$ boundary map. Let $1\leq p <2$ and $N \in \mathbb{N}$ be arbitrary. Then, for each $\e>0$, there exists $\psi \in C^\infty(\p\ball, \stwo)$ such that $\deg \psi =0$, $\|\varphi - \psi\|_{W^{1,p}} < \e$, $\mathcal{H}^2(\{\varphi \neq \psi\}) <\e$, and the Dirichlet integral has a unique minimiser over $W^{1,2}_\psi$ with at least $N$ singularities in $\ball$.
\end{proposition}

In contrast, R. Hardt and F.-H. Lin \cite{hl} proved that a minimising harmonic map is stable under Lipschitz perturbations of the boundary data, under an additional uniqueness assumption:
\begin{proposition}[The Stability Theorem in \cite{hl}]\label{propn: lipschitz}
Let $\Omega \subset \R^3$ be a smooth bounded domain and $\varphi \in {\rm Lip}(\p\Omega, \stwo)$. Suppose $v$ is the unique energy-minimising map from $\Omega$ to $\stwo$ with $v|\p\Omega =\varphi$. Then there exist a positive number $\beta >0$ and, for any $\e>0$, a positive number $\delta >0$, such that for any $\psi \in C^{1,\alpha}(\p\Omega, \stwo)$ with $\|\varphi - \psi\|_{\rm Lip} \leq \delta$ and any energy-minimising $u \in W^{1,2}(\Omega,\stwo)$ with $u|\p\Omega=\psi$, one has $\|u-v\circ\eta\|_{C^{0,\beta}} \leq \e$ for a bi-Lipschitz map $\eta :\Omega \map \Omega$ with $\|\eta - \id_\Omega\|_{\rm Lip} \leq \e$.
\end{proposition}

Our main result shows that, under the same assumptions of \cite{hl}, minimising harmonic maps are stable under $W^{1,p}$-perturbations of the boundary data for any $p\geq 2$. It demonstrates the sharpness of the index $p=2$ in Proposition \ref{propn: p<2}; Proposition \ref{propn: lipschitz} is the special case $p=\infty$.

\begin{theorem}\label{thm: W1,p, p>2}
	Let $\Omega \subset \R^3$ be a bounded Lipschitz domain and $\varphi \in W^{1,p}(\p\Omega,\stwo)$ for $p\geq 2$. Suppose $v$ is the unique energy-minimising map from $\Omega$ to $\stwo$ with $v|\p\Omega=\varphi$. Then there exist a positive number $\beta>0$ and, for any $\e>0$, a positive number $\delta>0$, such that for any $\psi \in C^{1,\alpha}(\p\Omega,\stwo)$ with $\|\varphi-\psi\|_{W^{1,p}} \leq \delta$ and any energy-minimising $u \in W^{1,2}(\Omega,\stwo)$ with $u|{\p\Omega}=\psi$, one has $\|u-v\circ\eta\|_{C^{0,\beta}} \leq \e$ for a bi-Lipschitz map $\eta :\Omega \map \Omega$ with $\|\eta - \id_\Omega\|_{\rm Lip} +\|\eta^{-1} - \id_\Omega\|_{\rm Lip} \leq \e$.
\end{theorem}

The arguments essentially follow \cite{hl} by Hardt--Lin. We remark that the uniqueness of $v$ is necessary: see $\S 5$, \cite{hl} for an example of a smooth boundary map that serves as boundary data for two minimisers from $\ball$ to $\stwo$, one with no singularity and the other with two singularities. Moreover, Almgren--Lieb \cite{al} proved that the boundary data with unique minimisers are dense in the $W^{1,2}$-topology.


\noindent
{\bf Notations.} For embedded surfaces $\Sigma \subset \R^3$, we write $\dd A$ for the surface measure on $\Sigma$, and $\sna$ for the projection of the Euclidean gradient on $\R^3$ to $T\Sigma$. In the spherical polar coordinates, we write $x=r\omega$ for $r=|x|, \omega = x/|x|\in\stwo$, the unit $2$-sphere. For an $m$-dimensional submanifold $M$ of $\R^n$, $|M|$ denotes the $m$-dimensional Hausdorff measure of $M$. We write $\ball(x,\rho)$ for an Euclidean $3$-ball with centre $x$ and radius $\rho$;  $\ball_\rho := \ball(0,\rho)$ and $\ball:=\ball_1$. For sets $E$ and $F$, we write $E\sim F$ for the set difference, and $\1_E$ for the indicator function on $E$. The norms $\|\bullet\|_{W^{1,p}}, \|\bullet\|_{\rm Lip}$ and $\|\bullet\|_{C^{0,\beta}}$ without explicitly indicating the domains are taken over the whole of $\Omega, \ball$ or $\stwo$, which will be clear from the context. $\bo(3)$ is the group of $3 \times 3$   orthogonal matrices.

\bigskip
\noindent
{\bf Acknowledgement}.
This work has been done during Siran Li's stay as a CRM--ISM postdoctoral fellow at Centre de Recherches Math\'{e}matiques, Universit\'{e} de Montr\'{e}al and Institut des Sciences Math\'{e}matiques, and as a G.~C. Evans Instructor at Rice University, Houston. The author would like to thank these institutions for providing nice working atmosphere. 

We are greatly indebted to Bob Hardt for many insightful discussions and continuous support. We thank Armin Schikorra for kind communications and insightful discussions.

\bigskip
\noindent
{\bf Note added}. Upon completion of the paper, the author was informed of the very nice work \cite{mms} by Mazowiecka--Mi\'{s}kiewicz--Schikorra, in which a generalisation of Hardt--Lin's stability theorem is obtained independently. In \cite{mms} the stability in $W^{1,2}$-norm is proved for minimising harmonic maps with trace in $W^{s,p}$ for $s\in ]1/2,1]$, $p \in [2,\infty[$ such that $ps \geq 2$, provided that the traces are $W^{s,p}$-close. This may be compared with Theorem \ref{thm: W1,p, p>2} above, in which we proved the stability in $C^{0,\beta}$-norm with traces in $C^{1,\alpha}$ being $W^{1,p}$-close. Additionally, in \cite{mms} Almgren--Lieb's linear law on the number of singularities is also extended to the case of $W^{s,p}$-traces.

\section{Uniform Boundary Regularity}\label{sec: bdry reg}

In this section, we establish the following
\begin{lemma}\label{lem: boundary reg}
There exist constants $0<e_0, \ell_0 \leq 1$ and $\rho_0=\rho_0(\ell_0,e_0)>0$ such that the following holds. Let $g:\R^2 \map \R$ be a Lipschitz map with $g(0)=0=|\na g(0)|$ and $\|g\|_{W^{1,\infty}} \leq \ell_0$. Denote by $\Omega_g := \{(x_1,x_2,x_3)\in\ball: x_3 < g(x_1,x_2)\}$. Assume that $u \in W^{1,2}(\Omega_g,\stwo)$ is an energy-minimising map with 
$\| u|\ball \cap \p\Omega_g \|_{W^{1,p}} \leq e_0$; $2 \leq p \leq \infty$. Then $\| u|\ball_{\rho_0} \cap {\Omega_g}\|_{C^{0,\beta}} \leq e_0$ for some $0<\beta<1$. \end{lemma}


The proof of Lemma \ref{lem: boundary reg}  follows from an adaptation of $\S \S 5.4, 5.5$,  Hardt--Kinderlehrer--Lin \cite{hkl} and $\S 2$, Hardt--Lin \cite{hl}, in both of which the boundary data are assumed to be Lipschitz. On the other hand,  if $\Omega_g$ is $C^\infty$ additionally, then we recover Corollary $2.5$, Almgren--Lieb \cite{al}.


We need to modify the arguments in \cite{hkl, hl2, al} to deal with the lower regularity assumptions for the boundary map and the domain. One useful result is Theorem $5.7$, Hardt--Lin \cite{hl2}: 
\begin{lemma}\label{lem: hardt-lin}
Let $m$ be a positive integer,  let $N$ be a smooth Riemannian manifold, and let $1 < p < \infty$. Denote by $\ball^+ := \{(x^1,\ldots, x^m\in\R^m: \sum_{i=1}^m |x^i|^2 <1, x^m >0)\}$. If $u_0 \in W^{1,p}(\ball^+, N)$ is a degree-$0$-homogeneous $p$-minimising harmonic map, and if $u_0$ is constant on $\p\ball^+ \cap \{x^m=0\}$, then $u_0$ is a  constant function.
\end{lemma}

We also recall the monotonicity formula: let $u$ be an energy-minimiser and $0<\sigma<\rho<\rho_0$ such that $\ball(y,\rho_0) \Subset \ball$. Then
\begin{equation}\label{monotonicity formula}
\frac{1}{\rho}\int_{\ball(y,\rho)} |\na u|^2\,\dd x - \frac{1}{\sigma}\int_{\ball(y,\sigma)} |\na u|^2\,\dd x =  \int_{\ball(y,\rho)\sim \ball(y,\sigma)} \frac{2}{r}\left|\frac{\p u}{\p r}\right|^2\,\dd x \geq 0.
\end{equation}
The proof follows by considering  ``squeeze deformations'' of $u$; {\it cf.} Lemma 2.5, \cite{su1}; Lemma 1.3, \cite{su2} and $\S 2.4$, \cite{s2}, among others.

\begin{proof}[Proof of Lemma \ref{lem: boundary reg}]
	By a standard blowup argument --- {\it cf}. $\S 5$ in  Hardt--Kinderlehrer--Lin \cite{hkl} --- it suffices to prove a uniform bound on the rescaled energy: for $\rho_0$ sufficiently small, there exists $c_0>0$ such that 
		\begin{equation}\label{normalised energy bound}
	\frac{1}{\rho_0} \int_{\ball_{2\rho_0} \cap \Omega_g} |\na u|^2\,\dd x \leq c_0.
	\end{equation}	
(One may conclude by choosing $c_0$ depending on $e_0$, and then shrinking $\rho_0$ if necessary.)

As in \cite{hkl}, \eqref{normalised energy bound} will follow from an absolute bound
\begin{equation}\label{abs bound}
\int_{\ball_{1/2} \cap \Omega_g} |\na u|^2\,\dd x \leq c_1,
\end{equation}
where $c_1$ depends only on $p$ and $\ell_0$. In particular, the arguments for ``energy decay/improvement'' in $\S \S 5.4, 5.5$, \cite{hkl} can be directly adapted to the case of $W^{1,p}$-boundary data. In the sequel let us exhibit a $c_1$.

For {\it a.e.} $\sigma\in [1/2,1]$, choose a bi-Lipschitz map $\Phi_\sigma: \ball_\sigma \cap \Omega_g \map \ball_\sigma$. The bi-Lipschitz constant of $\Phi_\sigma$ is universal; let us call it $\Lambda$. It depends only on $\|g\|_{W^{1,\infty}} \leq \ell_0$. We {\em claim} that there is $\omega_\sigma$, an extension of $(u \circ \Phi_\sigma^{-1})|\p\ball_\sigma$, that satisfies the following inequality:
\begin{equation}\label{caccioppoli}
\int_{\ball_\sigma} |\na \omega_\sigma|^2\,\dd x \leq c_2 \left\{ \int_{\p\ball_\sigma} \left|\sna (u \circ \Phi_\sigma^{-1})\right|^2\,\dd A \right\}^{1/2}
\end{equation}
for {\it a.e.} $\sigma \in [1/2,1]$. 

To see this, we follow the arguments in $\S 2.3$, \cite{hkl}. Let $\lambda=\lambda(\sigma)$ be the vector $|\ball_{\sigma}|^{-1}\int_{\ball_\sigma} (u \circ \Phi_\sigma^{-1})\,\dd x$ in $\R^3$. By Fubini's theorem, for {\em a.e.} $\sigma' \in [\sigma/2,\sigma]$  we have
\begin{eqnarray}
&&\int_{\p\ball_{\sigma'}} \left|\sna \left(u \circ \Phi_\sigma^{-1}\right)\right|^2\,\dd A \leq 8 \int_{\ball_\sigma}  \left|\na \left(u \circ \Phi_\sigma^{-1}\right)\right|^2\,\dd x,\label{cac1}\\
&&\int_{\p\ball_{\sigma'}} \left| \left(u \circ \Phi_\sigma^{-1}\right) - \lambda\right|^2\,\dd A \leq 8 \int_{\ball_\sigma} \left| \left(u \circ \Phi_\sigma^{-1}\right)- \lambda\right|^2\,\dd x.\label{cac2}
\end{eqnarray}
The right-hand sides of \eqref{cac1} and \eqref{cac2} are finite, thanks to 
\begin{equation*}
\int_{\ball_\sigma} \left|\na \left(u \circ \Phi_\sigma^{-1}\right)\right|^2\,\dd x \leq \Lambda^2 \int_{\Omega_g \cap \ball_\sigma} \left|\na u\right|^2\,\dd x
\end{equation*}
and the Poincar\'{e} inequality.

Let $h: \ball_{\sigma'} \map \R^3$ be the harmonic {\em function} --- {\it i.e.}, $\Delta h =0$ --- with $h|{\p\ball_{\sigma'}}= (u \circ \Phi_\sigma^{-1})|{\p\ball_{\sigma'}}$. By an elementary computation, all harmonic  functions fulfil  the identity
\begin{equation}\label{harm map identity}
\sigma' \int_{\p \ball_{\sigma'}} \left|\sna h\right|^2\,\dd A = \int_{\ball_{\sigma'}} |\na h|^2\,\dd x + \sigma' \int_{\p\ball_{\sigma'}} \left| \frac{\p h}{\p r}\right|^2\,\dd A.
\end{equation}
Thus, using integration by parts, the Cauchy--Schwarz inequality, \eqref{harm map identity},  and that $\sna h = \sna (u\circ\Phi^{-1}_\sigma)$ on $\p\ball_{\sigma'}$,  we deduce
\begin{align}\label{ineq for h}
\int_{\ball_{\sigma'}} |\na  h|^2 \,\dd x &= \int_{\p \ball_{\sigma'}} (h-\lambda) \cdot \frac{\p (h-\lambda)}{\p r} \,\dd A \nonumber \\
&\leq \left\{\int_{\p\ball_{\sigma'}} \left|\left(u \circ \Phi_\sigma^{-1}\right) - \lambda\right|^2\,\dd A\right\}^{1/2}  \left\{\int_{\p\ball_{\sigma'}} \left|\sna \left(u \circ \Phi^{-1}_\sigma\right)\right|^2\,\dd A\right\}^{1/2}.
\end{align}

Now let us modify $h$ to a function with range in $\stwo$ satisfying the same bound as in \eqref{ineq for h}. Denote by $\Pi_a: \R^3 \map \stwo$ the projection
\begin{equation*}
\Pi_a (x) := \frac{x-a}{|x - a|}.
\end{equation*} 
By Sard's theorem, $\Pi_a \circ h \in W^{1,2}(\ball_{\sigma'}, \stwo)$ for almost every $a \in \ball_{\sigma'/2}$. We have
\begin{align*}
|\na (\Pi_a \circ h )| = \Bigg| \frac{\na h}{|h-a|} - \frac{\na h \cdot (h-a) \otimes (h-a)}{|h-a|^3} \Bigg| \leq 2 \frac{|\na h|}{|h-a|}.
\end{align*}
Thus
\begin{align*}
\int_{\ball_{\sigma'/2}}\int_{\ball_{\sigma'}} \big|\na \big(\Pi_a \circ h(x)\big)\big|^2 \,\dd x \,\dd a &\leq 4\int_{\ball_{\sigma'}} |\na h(x)|^2 \left\{\int_{\ball_{\sigma'/2}} |h(x)-a|^{-2}\,\dd a\right\}\,\dd x\\
&\leq 4\pi  \int_{\ball_{\sigma'}} |\na h(x)|^2\dd x.
\end{align*}
In particular, by Fubini we can choose $a \in \ball_{\sigma'/2}$ such that 
\begin{equation*}
\int_{\ball_{\sigma'}} \big|\na \big(\Pi_a \circ h(x)\big)\big|^2 \,\dd x \leq 8\pi \int_{\ball_{\sigma'}} |\na h(x)|^2\dd x.
\end{equation*}
One thus deduces from \eqref{ineq for h} that
\begin{equation*}
\int_{\ball_{\sigma'}} \big|\na \big(\Pi_a \circ h(x)\big)\big|^2 \,\dd x \leq  8\pi \left\{\int_{\p\ball_{\sigma'}} \left|\left(u \circ \Phi_\sigma^{-1}\right) - \lambda\right|^2\,\dd A\right\}^{1/2}  \left\{\int_{\p\ball_{\sigma'}} \left|\sna \left(u \circ \Phi^{-1}_\sigma\right)\right|^2\,\dd A\right\}^{1/2}.
\end{equation*}
But $u$ takes values in $\stwo$; so
\begin{align*}
&\int_{\p\ball_{\sigma'}} \left|\left(u \circ \Phi_\sigma^{-1}\right) - \lambda\right|^2\,\dd A\\
&\quad \leq 2  \int_{\p\ball_{\sigma'}} \left|\left(u \circ \Phi_\sigma^{-1}\right)\right|^2\,\dd A + 2 \int_{\p\ball_{\sigma'}} \lambda^2\,\dd A \leq 4|\p\ball_{\sigma'}| \leq 16\pi;
\end{align*} 
hence 
\begin{equation}
\int_{\ball_{\sigma'}} \big|\na \big(\Pi_a \circ h(x)\big)\big|^2 \,\dd x \leq 32\pi^{3/2} \left\{\int_{\p\ball_{\sigma'}} \left|\sna \left(u \circ \Phi^{-1}_\sigma\right)\right|^2\,\dd A\right\}^{1/2}.
\end{equation}

Finally, set 
\begin{equation}
w_\sigma := (\Pi_a|\p\ball_{\sigma'})^{-1} \circ \Pi_a \circ h.
\end{equation}
By construction $w_\sigma|\p\ball_{\sigma'} = (u \circ \Phi_\sigma^{-1})|{\p\ball_{\sigma'}}$. The Lipschitz norm of $(\Pi_a|\p\ball_{\sigma'})^{-1}$  can be bounded geometrically as follows. For $a \in \ball_{\sigma'/2}$ given, set up the polar coordinate centred at $a$. Then $\|(\Pi_a|\p\ball_{\sigma'})^{-1}\|_{\rm Lip}$ equals the maximal ratio $\ell_{a,\sigma'}\slash \theta_a$, where $\theta_a$ is the angle between two straight lines emanating from $a$, and $\ell_{a,\sigma'}$ is the length of the arc $\mathscr{A}$ on $\p\ball_{\sigma'}$ swept out by such straight lines opening at angle $\theta_a$. By elementary  Euclidean geometry, the supremum over $a \in \ball_{\sigma'/2}$ of $\ell_{a,\sigma'}\slash \theta_a$ is  attained only if $a \in \p \ball_{\sigma'/2}$ and $\theta_a$ is bisected by the straight line through $a$ and $0$. In this case, $\ell_{a,\sigma'}\slash \theta_a = \sigma' \alpha \slash \theta_a$, where $\alpha$ is the angle formed by arc $\mathscr{A}$ and the origin. Clearly $\sigma' \alpha \slash \theta_a \leq 2 \sigma' \leq 2$; hence $$\left\|\left(\Pi_a|\p\ball_{\sigma'}\right)^{-1}\right\|_{\rm Lip} \leq 2.$$ We can thus conclude \eqref{caccioppoli} by choosing $c_2=128\pi^{3/2}$ (replacing $\sigma'$ with $\sigma$). 

Now, define
\begin{equation}
\D (\sigma) := \int_{\ball_\sigma \cap \Omega_g} |\na u|^2\,\dd x.
\end{equation}
By the minimality of $u$, we have
\begin{align*}
\D (\sigma) &\leq \int_{\ball_\sigma \cap \Omega_g} |\na (\omega_\sigma \circ \Phi_\sigma)|^2\,\dd x  \\
&\leq  \|\na \Phi_\sigma\|_{L^\infty}^2 \int_{\ball_\sigma} |\na \omega_\sigma|^2\,\dd x\\
&\leq c_2 \|\na \Phi_\sigma\|_{L^\infty}^2 \left\{ \int_{\p\ball_\sigma} \left|\sna \left(u \circ \Phi_\sigma^{-1}\right)\right|^2\,\dd A \right\}^{1/2}\\
&\leq c_2  \|\na \Phi_\sigma\|_{L^\infty}^2\left\|\na \Phi_\sigma^{-1}\right\|_{L^\infty} \left\{ \int_{\p\ball_\sigma \cap \Omega_g} \left|\sna u\right|^2\,\dd A + \int_{\ball_\sigma \cap \p\Omega_g} \left|\sna u\right|^2\,\dd A  \right\}^{1/2}.
\end{align*}
H\"{o}lder's inequality and the assumptions on $\|u|\ball\cap \p\Omega_g\|_{W^{1,p}}$ and $g$ give us
\begin{align}\label{qqq}
\int_{\ball_\sigma \cap \p\Omega_g} \left|\sna u\right|^2\,\dd A &\leq \left\{\int_{\p\Omega_g} \left|\sna u\right|^p\,\dd A\right\}^{\frac{2}{p}} |\ball_\sigma \cap \p\Omega_g |^{\frac{p-2}{p}} \nonumber\\
&\leq 1\times \left\{\int_{\{z\in\R^2: |z| \leq \sigma\}\cap\Omega_g} \sqrt{1+|\na g|^2}\,\dd z\right\}^{\frac{p-2}{p}} \leq \left(\sqrt{2}\pi \sigma^2\right)^{\frac{p-2}{p}}.
\end{align}
Thus, for {\it a.e.} $\sigma \in [1/2,1]$, with the previously chosen value of $c_2$ we have
\begin{equation}\label{ineq for D}
\D (\sigma) \leq 128 \pi^{3/2} \Lambda^3 \Big( {\D'(\sigma)} + (\sqrt{2}\pi \sigma^2)^{\frac{p-2}{p}}\Big)^{1/2}.
\end{equation}

To prove \eqref{abs bound}, it is enough to establish $\D(1/2) \leq c_1$. Let us write $c_1=\theta^{-1}$ and assume for contradiction that $\D(1/2) > \theta^{-1}$ for each $\theta>0$. Then
\begin{align*}
\int_{1/2}^1 \frac{-\D'(\sigma)}{\D^2(\sigma)}\,\dd \sigma = \frac{1}{\D(1)} - \frac{1}{\D(1/2)} > -\theta.
\end{align*}
On the other hand, by \eqref{ineq for D} there holds
\begin{align*}
\frac{\D'(\sigma)}{\D^2(\sigma)} \geq \left( \frac{1}{128 \pi^{3/2} \Lambda^3}\right)^2 - \frac{\left(\sqrt{2}\pi\right)^{\frac{p-2}{p}}}{\D(\sigma)^2} \geq \left( \frac{1}{128 \pi^{3/2}\Lambda^3}\right)^2 - \left(\sqrt{2}\pi\right)^{\frac{p-2}{p}}\theta^2.
\end{align*}
Integrating $\sigma$ over $[1/2,1]$, we get
\begin{align*}
\wp(\theta):=\left(\sqrt{2}\pi\right)^{\frac{p-2}{p}} \theta^2 + 2\theta - \frac{1}{16384 \pi^3 \Lambda^6} > 0.
\end{align*}
However, $\wp$ has a positive root $\theta_0>0$, so any $\theta \in ]0,\theta_0[$ would violate the above inequality. To be concrete, we can take $\theta = \theta_0/2$, {\it i.e.},
\begin{equation*}
c_1 ={2^{1+\frac{p-2}{2p}}\pi^{\frac{p-2}{p}}}
\left({\sqrt{1+ \frac{\left(\sqrt{2}\pi\right)^{\frac{p-2}{p}}}{16384\pi^3\Lambda^6}} -1}\right)^{-1},
\end{equation*}
where $\Lambda$ is the supremum of the bi-Lipschitz constant of $\Phi_\sigma$ over $\sigma \in [1/2,1]$. This gives the desired contradiction and thus concludes \eqref{abs bound}.

Finally, let us establish the bound \eqref{normalised energy bound}. If it were false, for some $c>0$ there would exist sequences of positive numbers $\{\rho_i\}\searrow 0$, $\{e_i\}\searrow 0$, and $\{\ell_i\} \searrow 0$, Lipschitz maps $\{g_i\}$ with $\|g_i\|_{W^{1,\infty}} \leq \ell_i$, and minimisers $\{u_i\} \subset W^{1,2}(\Omega_{g_i}, \stwo)$, such that \begin{equation}\label{contradiction}
\big\| u_i |{\ball \cap \p\Omega_{g_i}}\big\|_{W^{1,p}} \leq e_i\qquad \text{ but } \qquad\liminf_{i \map \infty}\frac{1}{\rho_i} \int_{\ball_{2\rho_i \cap \Omega_{g_i}}} |\na u_i|^2\,\dd x \geq c.
\end{equation}
Denote by 
\begin{equation*}
\tilde{u}_i(x) := u_i(2\rho_i x), \qquad \tilde{g}_i(x) := g_i(2\rho_i x).
\end{equation*}
Then $\|\tilde{g}_i\|_{W^{1,\infty}} \leq 2\rho_i\ell_i$ and
\begin{equation*}
\frac{1}{2\rho_i} \int_{\ball_{\rho_i} \cap \Omega_{g_i}} |\na u_i|^2\,\dd x = \int_{\ball_{1/2} \cap \Omega_{\tilde{g}_i}} |\na \tilde{u}_i|^2\,\dd x \leq c_1,
\end{equation*}
where $c_1$ is as in \eqref{abs bound}. As a result, a subsequence of $\{\tilde{u}_i\}$ converges weakly to $v \in W^{1,2}(\ball^+,\stwo)$. By monotonicity identity  \eqref{monotonicity formula}, $v$ is degree-$0$-homogeneous. Thanks to Theorem $6.4$ in  Hardt--Lin \cite{hl2}, the convergence $\tilde{u}_i \map v$ is indeed strong in the $W^{1,2}$-topology, and $v$ is a minimising harmonic map. But  the first inequality in \eqref{contradiction} implies that the limiting map $v \in W^{1,2}(\ball^+,\stwo)$ is constant on $\ball\cap\{x_3 = 0\}$, up to the choice of a representative in the Sobolev class. In view of Lemma \ref{lem: hardt-lin}, this contradicts the second inequality in \eqref{contradiction}. 

Hence the assertion follows.  \end{proof}

In the proof above, \eqref{qqq} and Lemma \ref{lem: hardt-lin} require $p\geq 2$. In fact, in view of the later parts of the paper and \cite{al, ms}, Lemma \ref{lem: boundary reg} is invalid for any $p<2$. 

\section{The Model Case: Stability of Hedgehog on $\Omega = \ball$}\label{sec: ball}

In this section we prove Theorem \ref{thm: W1,p, p>2} for the model case $\Omega=\ball$, $\varphi=\id_{\stwo}$ with $p>2$. The general case shall be  obtained by glueing these building blocks together in \S \ref{sec: general}, with modifications for the critical case $p=2$. Recall from the hypotheses of Theorem \ref{thm: W1,p, p>2} that the boundary map $\psi$ has $C^{1,\alpha}$-regularity; see \cite{mms} for results on $\psi$ with lower regularity.

We shall crucially rely on the result below due to L. Simon (see Theorem 1, \cite{s1} and the exposition \cite{s2}). 
\begin{proposition}\label{propn: simon}
Let $\Omega$ be an open subset of $\R^n$. Let $u \in W^{1,2}(\Omega,\stwo)$ be an energy-minimising map. Assume that $\Theta(x/|x|)$ is a tangent map of $u$ at $0$, where $\Theta \in \mathcal{O}(3)$ is a rotation. Then such $\Theta$ is unique. Moreover, there are uniform constants $\beta_0 \in ]0,1]$ and $c>0$ depending only on $\Omega$ such that for all $r>0$ sufficiently small, we have
\begin{align*}
\mathscr{A}(r):=\left\|\frac{\p}{\p r} u(r\bullet)\right\|_{C^1(\stwo)} + \|u(r,\bullet)-\Theta\|_{C^2(\stwo)} \leq cr^{\beta_0} \left\|u-\frac{x}{|x|}\right\|_{C^2(\ball_{2/3}\sim\ball_{1/3})}.
\end{align*}

\end{proposition}

Let us recall the notion of {\em tangent maps} (see $\S 3.1$, \cite{s2} for details). In the setting of Proposition \ref{propn: simon}, take $\ball(y, \rho_0) \Subset \Omega$, and for any $\rho \in ]0,\rho_0]$ define the blowup maps $$u_{y,\rho}(x):=u(y+\rho x).$$ Then, by the monotonicity formula \eqref{monotonicity formula}, there holds $\int_{\ball}|\na u_{y,\rho}|^2\,\dd x \leq \rho_0^{-1}\int_{\ball(y,\rho_0)} |\na u|^2\,\dd x$. By \cite{su1, hl2}, for any $\{\rho_j\}\searrow 0$ we can select a subsequence (not relabelled) $\{u_{y,\rho_j}\}$ that converges strongly in $W^{1,2}_{\rm loc}$ on $\R^n$ to an energy-minimiser $u_0$. Any $u_0$ thus obtained is called a tangent map of $u$ at $y$. The uniqueness of tangent maps remains a major open problem in the large. 

Proposition~\ref{propn: simon} is a consequence of Simon's asymptotic theory of nonlinear evolution equations developed in \cite{s1}. Indeed, Theorem~1 and Section~8 therein show that $\mathscr{A}(r)\to 0$ as $r \searrow 0$, provided that the target manifold $N$ is real-analytic. When specialising to $N=\stwo$, it follows from Brezis--Coron--Lieb \cite{bcl} that the tangent map in the proposition must be of the form $\Theta(x/|x|)$ for a rotation $\Theta$. In this case, the integrability of Jacobi fields (see Simon, \S 6 in \cite{s2} and Gulliver--White \cite{gw}) yields the desired decay estimate of $\mathscr{A}(r)$. Similar arguments were used in the proof of convergence to tangent cones of minimal submanifolds by Almgren--Allard in \cite{aa}.

\subsection{Singularity is Unique}\label{subsec: unique sing} Take $\Omega = \ball$ and  $\varphi=\id_{\stwo}$. Then $v:\ball\map\stwo$, the unique minimising map with $v|\p\Omega=\id_{\stwo}$, is the ``hedgehog'' $$v(x)=\frac{x}{|x|}$$
(see Brezis--Coron--Lieb \cite{bcl}). Assume for contradiction that a sequence $\{u_i\}\subset W^{1,2}(\ball,\stwo)$ is energy-minimising with boundary data $\psi_i := u_i | \p\ball \in C^{1,\alpha}(\p\ball,\stwo)$, so that $$\delta_i := \|\psi_i -\id_\stwo\|_{W^{1,p}} \longrightarrow 0$$ but $u_i$ has more than one singularity for large enough $i$. 


First, by the minimality of $u_i$, we get 
\begin{align*}
\int_\ball |\na u_i|^2 \,\dd x&\leq \int_\ball \left|\na\left\{\psi_i \left(\frac{x}{|x|}\right)\right\}\right|^2\,\dd x\\
&\leq \int_{\stwo} |\sna \psi_i(x)|^2\,\dd A\\
&\leq (1+\kappa) \int_\stwo \left|\sna \id_\stwo\right|^2\,\dd A + \Big(1+\frac{1}{\kappa}\Big) \int_\stwo |\sna (\psi_i - \id_\stwo)|^2\,\dd A
\end{align*}
for any small $\kappa>0$. In the last line we used the simple inequality $(a+b)^2 \leq (1+\kappa)a^2+(1+\kappa^{-1})b^2$. Moreover, it is well-known that $x/|x|$ has the quantised energy $8\pi$:
\begin{equation*}
\int_\stwo |\sna \id_\stwo|^2\,\dd A = \int_\ball \Big|\na\Big(\frac{x}{|x|}\Big)\Big|^2\,\dd x = 8\pi.
\end{equation*}
In addition,
\begin{align*}
\int_\stwo |\sna (\psi_i - \id_\stwo)|^2\,\dd A &\leq \|\psi_i-\id_\stwo\|^2_{{W}^{1,p}} \|\1_\stwo\|_{L^{\frac{p}{p-2}}} = (4\pi)^{\frac{p-2}{p}}  \|\psi_i-\id_\stwo\|^2_{{W}^{1,p}}.
\end{align*}
Thus
\begin{align}\label{comparison 1}
\int_\ball |\na u_i|^2 \,\dd x \leq (1+\kappa) 8\pi + \Big(1+\frac{1}{\kappa}\Big){(4\pi)^{\frac{p-2}{p}} } (\delta_i)^2.
\end{align}

Thanks to the $W^{1,2}$-bound in \eqref{comparison 1}, $\{u_i\}$ has a subsequence (not relabelled) that converges weakly in $W^{1,2}$. By sending first $i \nearrow \infty$ and then $\kappa \searrow 0$, any such limit function has energy $\leq 8\pi$ and boundary map $\id_\stwo$. Again by Brezis--Coron--Lieb \cite{bcl}, it must be $x/|x|$. Using the arguments by Schoen--Uhlenbeck (\cite{su1}, also see L. Simon \cite{s2} via Luckhaus' lemma \cite{l}), we have
\begin{equation}\label{strong convergence}
u_i(x) \longrightarrow \frac{x}{|x|}\qquad \text{ strongly in } W^{1,2}.
\end{equation}

Now, in view of Lemma \ref{lem: boundary reg}, there exists a universal $\rho_0>0$ such that $u_i$ are uniformly H\"{o}lder continuous with uniformly bounded energy on some neighbourhood of $\p(\ball \sim \ball_{1-\rho_0})$. By the definition of $\delta_i$, $\deg(\psi_i)$ is equal to $1$ for sufficiently large $i$. So the singular set $sing\,u_i$ is non-empty and lies in $\ball_{1-\rho_0}$, {\it i.e.}, away from the boundary $\p\ball$. As $x/|x|$ is H\"{o}lder continuous away from $0$, thanks to \eqref{strong convergence} and the interior regularity result in \cite{su1}, we may conclude that the diameter of  $sing\, u_i$ tends to zero.

For any $r \in ]0,1/20[$, there is $i$ large enough such that $\ball_{1-|a_i|} \subset \ball_{1-\rho_0}$, $|a_i| < r/4$ for every $a_i \in sing\, u_i$. Consider $\bar{u}_i (x):=u_i(x+a_i)$ defined on $\ball_{1-|a_i|}$. Then we have
\begin{align*}
&\left\| \bar{u}_i - \frac{x}{|x|}\right\|_{C^2(\ball(a_i,1/2) \sim \ball_{r/2})} \\
&\quad \leq \left\| u_i - \frac{x}{|x|}\right\|_{C^2(\ball_{1-2\rho_0} \sim \ball_{1/10})} + \left\| \frac{x-a_i}{|x-a_i|} - \frac{x}{|x|}\right\|_{C^2(\ball_{1-3\rho_0} \sim  \ball_{1/5})} \longrightarrow 0.
\end{align*}
The convergence of the first term follows from the interior regularity theory (see Schoen--Uhlenbeck \cite{su1}), and the convergence of the second term can be deduced from direct computation. Using the asymptotic theory of Simon (Proposition~\ref{propn: simon}), we have $sing\,\bar{u}_i = \{0\}$ for sufficiently large $i$. This contradicts the assumption that $u_i$ has more than one singularities.

Therefore, there exists $\delta>0$ such that for any $\psi \in C^{1,\alpha}(\p\ball,\stwo)$ with $\|\psi - \id_\stwo\|_{W^{1,p}} \leq \delta$, any minimiser $u$ with $u|\p\ball=\psi$ has a unique singular point.

In the sequel we say $sing\, u = \{a\}$.

\subsection{Modulus of  Singularity}\label{subsec: modulus of sing} 
To estimate the modulus $|a|$, we pick some $\rho \in ]0,1[$ and define 
\begin{equation}
w(x):= \begin{cases}
u(\rho^{-1}x),\qquad 0\leq |x|<\rho,\\
{z(x)}/{|z(x)|},\qquad \rho \leq |x| \leq 1,
\end{cases}
\end{equation}
 where
 \begin{equation*}
 z(x) := \frac{1}{1-\rho} \left\{ (1-|x|) \psi\left(\frac{x}{|x|}\right) + (|x|-\rho)\frac{x}{|x|} \right\}.
 \end{equation*}
 
In $\ball_\rho$ there holds
\begin{align*}
\int_{\ball_\rho} |\na w(x)|^2\,\dd x= \rho\int_\ball |\na u(y)|^2\,\dd y.
\end{align*}
For $x \in \ball \sim\ball_\rho$, we shall estimate by
 \begin{align*}
 \int_{\ball\sim\ball_\rho} |\na w|^2\,\dd x = \int_{\ball\sim\ball_\rho} \left\{ \frac{|z|^2|\na z|^2 - |z\cdot \na z|^2}{|z|^4}  \right\}\,\dd x \leq \int_{\ball\sim\ball_\rho} \frac{|\na z|^2}{|z|^2}\,\dd x.
\end{align*}
Notice that
 \begin{equation*}
 z(x)-\xx = \frac{1-|x|}{1-\rho} (\psi-\id_\stwo)\Big(\xx\Big);
 \end{equation*}
so for $\rho \leq |x| \leq 1$ one has
\begin{equation}\label{sobolev}
|z(x)| \geq 1-\left| \frac{1-|x|}{1-\rho} \right|\|\psi-\id_\stwo\|_{L^\infty} \geq 1-c_5\delta,
\end{equation}
where $c_5=c(p)$ is the Sobolev constant for $W^{1,p}(\p\ball,\stwo) \emb C^0(\p\ball,\stwo)$ for $p>2$. Hence
 \begin{align*}
 \int_{\ball\sim\ball_\rho} |\na w|^2\,\dd x \leq \int_{\ball\sim\ball_\rho} \frac{|\na z(x)|^2}{(1-c_5\delta)^2}\,\dd x.
 \end{align*}
But \begin{equation*}
 \na z(x) - \na \left(\xx\right) = \frac{1}{1-\rho}\left\{ \xx \otimes (\id_\stwo -\psi) \left(\xx\right) + (1-|x|) \left[ \na \psi\left(\xx\right) - \na \left(\xx\right) \right] \right\};
 \end{equation*}
so, computing in spherical polar coordinates using $(a+b)^2 \leq (1+\kappa)a^2+(1+\kappa^{-1})b^2$ and H\"{o}lder's inequality, we get
\begin{align*}
 \int_{\ball\sim\ball_\rho} |\na z|^2\,\dd x &\leq (1+\kappa) 
 \int_{\ball\sim\ball_\rho} \Big|\na \Big(\xx\Big)\Big|^2\,\dd x + (1+\kappa^{-1}) 
 \int_{\ball\sim\ball_\rho} \bigg\{\frac{1}{1-\rho}\Big| \xx - \psi\Big(\xx\Big)\Big|\bigg\}^2\,\dd x \\
 &\quad  + (1+\kappa^{-1}) 
 \int_{\ball\sim\ball_\rho} \bigg\{\frac{1-|x|}{1-\rho}\bigg| \na\Big(\psi\big(\xx\big)\Big) - \na \Big(\xx\Big) \bigg|\bigg\}^2\,\dd x\\
 & \leq  (1+\kappa) 
 \int_{\ball\sim\ball_\rho} \Big|\na \Big(\xx\Big)\Big|^2\,\dd x  + \frac{1+\kappa^{-1}}{(1-\rho)^2} \Bigg\| \Big|\xx-\psi\Big(\xx\Big) \Big|^2 \Bigg\|_{L^{p/2}(\ball\sim\ball_\rho)}
|\ball\sim\ball_\rho|^{\frac{p-2}{p}} \\
 &\quad + \frac{1+\kappa^{-1}}{(1-\rho)^2} \Bigg\| \bigg| \na\Big(\psi\big(\xx\big)\Big) - \na \Big(\xx\Big) \bigg|^2\Bigg\|_{L^{p/2}(\ball\sim\ball_\rho)} \big\|(1-|x|)^2\big\|_{L^{\frac{p}{p-2}}(\ball\sim\ball_\rho)}\\
&\leq (1+\kappa) 
 \int_{\ball\sim\ball_\rho} \Big|\na \Big(\xx\Big)\Big|^2\,\dd x + \frac{(1+\kappa^{-1})(1-\rho^3)}{3(1-\rho)^2}(4\pi)^{\frac{p-2}{p}} \delta^2.
\end{align*}

Putting the above estimates together, we arrive at
\begin{align}\label{comparison 2}
\int_\ball |\na w|^2\,\dd x & \leq \rho \int_\ball |\na u|^2\,\dd x +  \frac{1+\kappa}{(1-c_5\delta)^2} 
 \int_{\ball\sim\ball_\rho} \Big|\na \Big(\xx\Big)\Big|^2\,\dd x  + c_6 \frac{1+\kappa^{-1}}{(1-c_5\delta)^2} \delta^2,
\end{align}
where $c_6$ depends only on $p$ (via the Sobolev constant $c_5$) and $\rho$.

Now, as the topological degree of $u$ on $\p\ball_s$ is $1$ for each $s \in [\rho,1]$, we have
\begin{align}\label{new 1}
 \int_{\ball\sim\ball_\rho} \Big|\na \Big(\xx\Big)\Big|^2\,\dd x &= \int_\rho^1\int_{\p\ball_s}\Big|\na \Big(\xx\Big)\Big|^2  \,\dd \mathcal{H}^2 \,\dd s = 8\pi (1-\rho)\nonumber\\
 &\leq 2\int_\rho^1 \big| u(\p\ball_s) \big|\,\dd s = \int_\rho^1 \int_{\p\ball_s} |\sna u|^2\,\dd \mathcal{H}^2\,\dd s
\end{align}
by the area formula. Therefore, using \eqref{comparison 2}\eqref{new 1} and the monotonicity formula \eqref{monotonicity formula}, one deduces
\begin{align}\label{comparison 3}
\int_\ball |\na w|^2\,\dd x &\leq  \rho\int_{\ball}|\na u|^2\,\dd x  +\int_\rho^1\int_{\p\ball_s} |\sna u|^2\,\dd \mathcal{H}^2\,\dd s \nonumber\\
&\qquad + \bigg\{\frac{1+\kappa}{(1-c_5\delta)^2}-1\bigg\} \int_{\ball\sim\ball_\rho}\bigg|\na\Big(\frac{x}{|x|}\Big)\bigg|^2\,\dd x + c_6 \frac{1+\kappa^{-1}}{(1-c_5\delta)^2} \delta^2 \nonumber \\
&\leq \int_\ball |\na u|^2\,\dd x + 8\pi (1-\rho) \bigg\{ \frac{1+\kappa}{(1-c_5\delta)^2} -1 \bigg\} + c_6 \frac{1+\kappa^{-1}}{(1-c_5\delta)^2} \delta^2
\end{align}
for each $p>2$, $0<\rho<1, \kappa>0$ and sufficiently small $\delta>0$. 

On the other hand, as $w|\p\ball=\id_\stwo$ and $sing\, w =\{a\}$, the estimates by Brezis--Coron--Lieb (\cite{bcl}; also see the last inequality on p.117, \cite{hl}) lead to
\begin{equation}\label{lower bound}
\int_\ball |\na w|^2\,\dd x \geq 8\pi + c_7 |a|^2 
\end{equation}
with a universal constant  $c_7$. Furthermore, the estimate \eqref{comparison 1} holds with $u_i$, $\delta_i$ replaced by $u$ and $\delta$, respectively. Combining with \eqref{comparison 3} and \eqref{lower bound}, we get
\begin{align}\label{comparison 5}
c_7|a|^2 &\leq 8\pi\kappa + \Big(1+\frac{1}{\kappa}\Big){(4\pi)^{\frac{p-2}{p}} } \delta^2 \nonumber \\
&\qquad + 8\pi (1-\rho) \bigg\{ \frac{1+\kappa}{(1-c_5\delta)^2} -1 \bigg\} + c_6 \frac{1+\kappa^{-1}}{(1-c_5\delta)^2} \delta^2
\end{align}

For each $\rho \in ]0,1[$ fixed, the penultimate term on the right-hand side of \eqref{comparison 5} satisfies 
\begin{align*}
 c_8 \Big\{ \kappa + 2c_5\delta + \bo(\delta^2) \Big\}\qquad \text{ as } \delta \searrow 0,
\end{align*}
where $c_8=8\pi (1-\rho)$. Also, for $0<\kappa,\delta\ll 1$, there exists $c_9=c(\rho, p)$ such that the final term of \eqref{comparison 5} can be bounded as follows:
\begin{align*}
 c_6 \frac{1+\kappa^{-1}}{(1-c_5\delta)^2} \delta^2 \leq c_9 \kappa^{-1} \delta^{2}.
\end{align*}
The optimal $\kappa>0$ we may choose is of order $\bo (\delta)$. We thus conclude from \eqref{comparison 5} that 
\begin{equation}\label{modulus of a = sing}
|a| \leq c_{10} \sqrt{\delta}
\end{equation}
for all $\delta \leq \delta_0$, where $\delta_0=c(\rho,p)>0$ is sufficiently small and $c_{10}=c(\rho, p)$.

From now on,  let us fix the parameter $\rho \in ]0,1[$.

\subsection{$W^{1,p}$-Stability for $x/|x|$ for $p >2$}\label{subsec completion of perturbation lemma}
As proved earlier in this section, $u$ has a unique singularity $a$, whose norm is controlled by $\sqrt{\delta}$ with $\|\psi-\id_\stwo\|_{W^{1,p}} \leq \delta$ and $u|\p\ball = \psi \in C^{1,\alpha}(\p\ball,\stwo)$. Here $u$ satisfies the assumptions of Theorem \ref{thm: W1,p, p>2} with $\Omega=\ball$ and $\varphi = \id_\stwo$; in particular, it is a minimising harmonic map.

Several consequences can be deduced (see p.118, \cite{hl}) ---

{\bf (i)} By $\S$ \ref{subsec: unique sing} and \cite{bcl} we have the quantisation of energy:
\begin{equation}\label{quantised energy}
\limsup_{r \searrow 0} \frac{1}{r} \int_{\ball(a,r)} |\na u|^2\,\dd x = 8\pi
\end{equation}
where $a$ is the singularity of $u$.

{\bf (ii)} The tangent map of $u$ at $a$ is unique (by Proposition \ref{propn: simon}) and takes the form $\T(x/|x|)$ with $\T \in O(3)$ (by Corollary $7.12$, Brezis--Coron--Lieb \cite{bcl}). 

{\bf (iii)} By Proposition \ref{propn: simon}, there are universal constants $\beta_0 \in ]0,1]$ and $c_{11} >0$ such that for $r>0$ sufficiently small,
\begin{equation}\label{new 2}
\left\|\frac{\p}{\p r} \bar{u}(r\bullet)\right\|_{C^1(\stwo)} + \left\|\bar{u}(r\bullet) - \Theta\right\|_{C^2(\stwo)}  \leq c_{11}\ee r^{\beta_0}.
\end{equation}
Specifically, for any $\alpha \in ]0,\beta_0[$ one has
\begin{equation}\label{B 1/2}
\left\| u - \T \left(\frac{x-a}{|x-a|}\right)\right\|_{C^{0,\alpha}(\ball_{1/2})} \leq c_{11}\ee.
\end{equation}
Here, for $\bar{u}:\ball_{1-|a|}\map\stwo$ and $\bar{u}(x):=u(x+a)$ we set
\begin{equation}\label{E}
\ee := \left\| \bar{u}-\xx \right\|_{C^2(\ball_{2/3}\sim \ball_{1/3})}.
\end{equation}

{\bf (iv)} By \cite{gw, s1, s2} there is a universal constant $c_{12}$ such that
\begin{equation}\label{matrix norm close}
\|\T - \id_{\R^3}\| \leq c_{12}\ee;
\end{equation}
here $\|\bullet\|$ denotes the matrix norm.

Having summarised {\bf (i)}--{\bf (iv)} above, let us proceed as follows.

 First, on the boundary $\p\ball$, there holds
\begin{equation*}
\Big\| \psi - \T \Big(\frac{x-a}{|x-a|}\Big)\Big\|_{W^{1,p}(\p\ball)} \leq \|\psi - \id_\stwo\|_{W^{1,p}(\p\ball)} + c_{12}\ee + \Big\|\frac{x-a}{|x-a|} - \frac{x}{|x|}\Big\|_{W^{1,p}(\p\ball)}.
\end{equation*}
But
\begin{align}\label{new 3}
\na\Big(\frac{x-a}{|x-a|}\Big) - \na \xx = \delta_{ij} \Big(\frac{1}{|x-a|} - \frac{1}{|x|} \Big) + \frac{(x-a)\otimes (x-a)}{|x-a|^3} - \frac{x\otimes x}{|x|^3},
\end{align}
thus a direct computation using $|x|=1$, $|a| \leq c_{10}\sqrt{\delta}$ yields 
\begin{equation}\label{add,1}
\Big\| \psi - \T \Big(\frac{x-a}{|x-a|}\Big)\Big\|_{W^{1,p}(\p\ball)} \leq \delta + c_{12}\ee + c_{13} \sqrt{\delta}
\end{equation}
for $c_{13}=c(p)$. 

Next, thanks to \eqref{new 2}\eqref{matrix norm close}, we have
\begin{align*}
\Big\| \psi - \T \Big(\frac{x-a}{|x-a|}\Big)\Big\|_{W^{1,p}(\p\ball_{1/2})} \leq c_{14}\ee + \Big\|\frac{x-a}{|x-a|} - \frac{x}{|x|}\Big\|_{W^{1,p}(\p\ball_{1/2})},
\end{align*}
where $c_{14}=c(\beta_0)$ with the universal constant $\beta_0$ in {\bf (iii)}. Taking $|x|=1/2$ in \eqref{new 3}, one obtains
\begin{align}\label{add,2}
\Big\| \psi - \T \Big(\frac{x-a}{|x-a|}\Big)\Big\|_{W^{1,p}(\p\ball_{1/2})} \leq c_{14}\ee +c_{15}\sqrt{\delta}
\end{align} 
for a universal constant $c_{15}$.

In what follows let us bound $\ee$ by a power of $\delta$. Then, choosing $\delta_0$ sufficiently small, for any $\delta \in ]0,\delta_0]$ we may apply the interior regularity theory (\cite{su1}) and Lemma \ref{lem: boundary reg} to deduce from \eqref{add,1}, \eqref{add,2} that
\begin{equation}\label{B 1- 1/2}
\Big\| u - \T \Big(\frac{x-a}{|x-a|}\Big)\Big\|_{C^{0,\alpha}(\ball \sim \ball_{1/2})} \leq c_{16}(\ee + \sqrt{\delta}).
\end{equation}
Here $c_{16}=c(p)$ is determined from $c_{12}, \ldots, c_{15}$ (one may shrink $\alpha \in ]0,\beta_0[$ if necessary to make it smaller than the universal constant $\beta$ in Lemma \ref{lem: boundary reg}). The desired bound for $\ee$ is achieved by adapting the arguments on pp.119--120, \cite{hl}.

To this end, we first notice that
\begin{equation}\label{E1}
\ee \leq J_1 + J_2:= \Big\|u-\xx\Big\|_{C^2(\ball_{3/4} \sim \ball_{1/4})} + \Big\|\xx - \frac{x-a}{|x-a|}\Big\|_{C^2(\ball_{2/3} \sim \ball_{1/3})},
\end{equation}
where
\begin{equation}\label{E2}
J_2 \leq c_{17}|a|,\qquad J_1 \leq c_{17} B.
\end{equation}
By interior regularity, $B$ can be chosen as an upper bound for the $L^2$-norm of $(u-x/|x|)$ in the larger annulus $\ball \sim \ball_{1/8} \supset \ball_{3/4}\sim \ball_{1/4}$; the constant $c_{17}=c(p)$.

Then, write $x = r\omega$ for $r=|x| \in [1/8,1]$, $\omega=x/|x|\in\stwo$; we have
\begin{align*}
& \int_{\ball \sim \ball_{1/8}} \Big| u(x) - \xx\Big|^2\,\dd x\\
&\quad \leq 2\int_{1/8}^1 \int_\stwo \bigg\{|u(r\omega) - \psi(\omega)|^2 + |\psi(\omega)-\omega|^2\bigg\}r^2\,\dd A(\omega)\,\dd r =: J_{11} + J_{12}.
\end{align*}
An application of H\"{o}lder's inequality yields
\begin{align*}
J_{12} &= 2\left(\int^1_{1/8} r^2\,\dd r\right) \int_\stwo |\psi - \id_\stwo|^2\,\dd A \\
&\qquad\qquad\leq \frac{2}{3}\left(1-\frac{1}{8^3}\right) \|\psi-\id_\stwo\|^2_{L^p} \left|\stwo\right|^{\frac{p-2}{p}} \leq c_{18} \delta^2,
\end{align*}
and a direct computation gives us
\begin{align*}
J_{11} &= 2\int_{1/8}^1 \int_\stwo \left|\int_r^1 \frac{\p u}{\p r}(s\omega)\,\dd s\right|^2r^2\,\dd A(\omega)\,\dd r\\
&\qquad\qquad\leq 2\int_{1/8}^1 \left\|\frac{\p u}{\p r}\right\|^2_{L^2(\ball\sim\ball_{r})}(1-r)\,\dd r \leq c_{19} \left\|\frac{\p u}{\p r}\right\|^2_{L^2(\ball\sim\ball_{1/8})},
\end{align*}
where $c_{18}=c(p)$ and $c_{19}$ is a universal constant. But $\|\p u/\p r\|_{L^2(\ball\sim\ball_{1/8})}$ can be controlled by the monotonicity formula \eqref{monotonicity formula} and the quantisation of energy \eqref{quantised energy}:
\begin{align*}
\left\|\frac{\p u}{\p r}\right\|^2_{L^2(\ball\sim\ball_{1/8})} \leq \int_\ball |\na u|^2\,\dd x - (1-8|a|)8\pi.
\end{align*}
Furthermore, recall from \eqref{comparison 1}:
\begin{equation*}
\int_\ball |\na u|^2\,\dd x - 8\pi \leq 8\pi \kappa + \left(1+\frac{1}{\kappa}\right){(4\pi)^{\frac{p-2}{p}}} \delta^2.
\end{equation*}
Putting together the above estimates, one obtains
\begin{align}\label{comparison 4}
 \int_{\ball \sim \ball_{1/8}} \left| u(x) - \xx\right|^2\,\dd x \leq c_{18}\delta^2 + c_{19}\left\{64\pi |a| + 8\pi \kappa + \left(1+\frac{1}{\kappa}\right)(4\pi)^{\frac{p-2}{p}} \delta^2 \right\}.
\end{align}
In view of \eqref{modulus of a = sing}, the best decay rate of the right-hand side of \eqref{comparison 4} is $\bo (\sqrt{\delta})$ --- {\it e.g.}, by choosing $\kappa =\bo({\delta})$. 

Therefore, taking the square root of \eqref{comparison 4} and utilising \eqref{E1}, \eqref{E2}, and \eqref{modulus of a = sing}, we can choose $\delta_0 >0$ sufficiently small such that, for $0<\delta \leq \delta_0$, there holds
\begin{equation*}
\ee \leq c_{20} \delta^{\frac{1}{4}}.
\end{equation*}
The constant $c_{20}=c(p)$. Moreover, by \eqref{B 1- 1/2} and \eqref{B 1/2}, for any sufficiently small $\alpha >0$ we have 
\begin{equation}\label{xx}
\left\| u - \T \left(\frac{x-a}{|x-a|}\right)\right\|_{C^{0,\alpha}(\ball)} \leq c_{21} \delta^{\frac{1}{4}}\qquad \text{ where } c_{21} = c(p).
\end{equation}

In summary, we obtain the following analogue of the Perturbation Lemma in \cite{hl}:
\begin{lemma}\label{lemma: perturbation}
Let $\psi \in C^{1,\alpha}(\p\ball,\stwo)$, $2<p\leq \infty$ and $\delta:=\|\psi-\id_\stwo\|_{W^{1,p}}$. There are positive constants $\delta_0$ and $c$ (depending on $p$) and $\alpha \in ]0,1[$, such that for any $\delta \in ]0, \delta_0]$ and $u \in W^{1,2}(\ball,\stwo)$ minimising the Dirichlet energy with $u|\p\ball = \psi$, one has $${\rm sing}\, u = \{a\}, \qquad |a| \leq c \sqrt{\delta}, \qquad \text{ and } \qquad \left\| u - \T \left(\frac{x-a}{|x-a|}\right)\right\|_{C^{0,\alpha}(\ball)} \leq c \delta^{1/4},$$
where $\T \in \bo(3)$ with $\|\T - \id_{\R^3}\|\leq c\delta^{1/4}.$
\end{lemma}

\section{Proof of Theorem \ref{thm: W1,p, p>2}}\label{sec: general}

\subsection{The case $2<p\leq \infty$}\label{subsec: 2<p}
With Lemma \ref{lemma: perturbation} at hand, Theorem \ref{thm: W1,p, p>2} follows as in $\S 3$ of \cite{hl} for the case $p > 2$. To be self-contained we sketch the arguments below.

Assume $u_i: \Omega \map \stwo$ are energy-minimisers with $u_i|\p\Omega = \psi_i \in C^{1,\alpha}(\p\ball,\stwo)$, such that $\|\psi_i-\varphi\|_{W^{1,p}(\stwo)} \map 0$ as $i \nearrow \infty$, and that $v:
\Omega \map \stwo$ is the {\em unique} minimiser with $v|\p\Omega = \varphi$. Then $\int_\Omega |\na u_i|^2\,\dd x$ is bounded ({\it e.g.}, by comparing with the harmonic extensions of $\psi_i$ and the uniform bound on $\|\psi_i\|_{W^{1,p}(\stwo)}$, $p>2$), $u_i \map v$ strongly in $W^{1,2}$ (by Theorem $6.4$, \cite{hl2}), and $sing\, v$ is a finite set (by Theorem 2, \cite{su1}) --- call it $\{a_j\}_{j=1}^k \subset \Omega$.

As before, the tangent map of $v$ at each $a_j$ is unique and equals $\T_j(x/|x|)$ for $\T_j \in \bo(3)$. For $0<\tau<\min\{{\rm dist} (a_j, ({\rm sing}\, v \sim \{a_j\})\cup\p\Omega )\}/2$, we have
\begin{equation}\label{add,3}
\|u_i-v\|_{W^{1,p}(\p\ball(a_j,\tau))} \longrightarrow 0 
\end{equation}
thanks to Simon's asymptotic theory (Proposition \ref{propn: simon}; also see \cite{s1, s2}) and a standard compactness argument. 

Denote by $\delta_i$ the larger of $\|\psi_i-\varphi\|_{W^{1,p}(\stwo)}$ and $\|u_i-v\|_{W^{1,p}(\p\ball(a_j,\tau))}$. Utilising the interior regularity theory (\cite{su1}) and the uniform boundary regularity Lemma \ref{lem: boundary reg}, one may infer that
\begin{equation}\label{new 4}
\|u_i - v\|_{C^{0,\alpha}(\Omega \sim \bigcup_{1\leq j \leq k}\ball(a_j,\tau))} \leq c_{22} \delta_i.
\end{equation} 
This gives the desired stability of minimisers away from the singularities of the limiting map.

Now, apply the arguments in $\S \ref{sec: ball}$ to each $\ball(a_j, \tau)$, $1\leq j \leq k$ and $u_i$ for large enough $i$. For each pair $(i,j)$, there exists a unique point $a_{ji} \in \ball(a_j,\tau)$ such that $sing\,u_i=\{a_{ji}\}$. Moreover, there are rotations $\T_{ji} \in \bo(3)$ so that
\begin{equation}
\sup_{1\leq j \leq k}\bigg\{|a_{ji}-a_j| + \|\T_{ji}-\T_j\| + \Big\| u_i -\T_{ji}\Big(\frac{x-a_{ji}}{|x-a_{ji}|}\Big)\Big\|_{C^{0,\alpha}(\ball(a_j,\tau))}\bigg\} \leq c_{23} \delta_i^{1/4}.
\end{equation}
Also, set
\begin{equation*}
\tau_i := \max_{1\leq j \leq k}|a_{ji}-a_j|^{1/2} \leq c_{24} \delta_i^{1/8}.
\end{equation*}

Finally, we construct the bi-Lipschitz homeomorphism $\eta: \Omega \map \Omega$ such that some H\"{o}lder norm of $(u_i - v \circ \eta)$ and $\|\eta - \id_\Omega\|_{\rm Lip}+\|\eta^{-1} - \id_\Omega\|_{\rm Lip}$ are both made arbitrarily small. Define $\eta_i$ for each $i$, such that $\eta_i = \id$ away from $sing\, v$, and near each $a_j$, $\eta_i$ maps $a_{ji}$ (the singularity of $u_i$) to $a_j$. In between, $\eta_i$ is connected by a smooth bump function. Then we take $\eta=\eta_i$ for large enough $i$.  More precisely, as on p.121, \cite{hl} we set
\begin{equation*}
\eta_i := \begin{cases}
\id\qquad\qquad\qquad\qquad\qquad\quad \text{ on } \Omega \sim \cup_{j=1}^k \ball(a_j, \tau_i),\\
\lambda_{ji} \xi_{ji} + (1-\lambda_{ji})\id\qquad\qquad \text{ on } \ball(a_j, \tau_i) \text{ for each } 1\leq j \leq k.
\end{cases}
\end{equation*}
Here $\xi_{ji}(x) = \T_j^{-1}\T_{ji} (x-a_{ji}) + a_j$ and $\lambda_{ji} \in C^\infty(\Omega, [0,1])$, $\lambda_{ji} \equiv 1$ on $\ball(a_j, \tau_i/2)$, $\lambda_{ji} \equiv 0$ on $\Omega \sim \ball(a_j, \tau_i)$ and $|\na \lambda_{ji}| \leq 2\tau_i$. Then, for sufficiently large $i$ and $\alpha' < \alpha/10$, we have 
\begin{equation}
\|\eta^{-1}_i - \id_\Omega\|_{\rm Lip}+\|\eta_i - \id_\Omega\|_{\rm Lip} \leq c_{24} \delta_i^{1/8},\qquad \|u_i - v \circ \eta_i\|_{C^{0,\alpha'}} \leq c_{24} \delta_i^{1/4}.
\end{equation}

This completes the proof of Theorem \ref{thm: W1,p, p>2} for $p>2$.

\subsection{The case $p=2$}
Now let us modify the preceding arguments to deal with the critical case $p=2$.  The uniform boundary regularity Lemma \ref{lem: boundary reg} holds for $p=2$, and the only place we used $p>2$ is the Sobolev--Morrey embedding \eqref{sobolev}. So we just need to modify the arguments in $\S \ref{sec: ball}$. 

Indeed, as the boundary maps $\psi, \id_\stwo:\p\ball \map \stwo$ take values in the unit sphere, for $\psi \in C^{1,\alpha}(\p\ball,\stwo)$ we have $\|\psi-\id_\stwo\|_{W^{1,\infty}(\stwo)} \leq c_{25}$, which depends only on the Lipschitz norm of $\psi$. Thus, applying the interpolation inequality
\begin{equation*}
\|f\|_{L^q} \leq \|f\|^{2/q}_{L^2} \|f\|^{1-2/q}_{L^\infty},\qquad q>2
\end{equation*}
to $f=\psi-\id_\stwo$ and $f=\sna \psi - \sna\id_\stwo$, we can find a constant $c_{26}=c(q,\|\psi\|_{\rm Lip})$ such that
\begin{equation}
\|\psi - \id_\stwo\|_{W^{1,q}(\stwo)} \leq c_{24} \delta^{2/q} =: \bar{\delta},
\end{equation}
whenever $q \in ]2,\infty[$ and 
\begin{equation}
\|\psi - \id_\stwo\|_{L^2(\stwo)} \leq \delta.
\end{equation}

Now, one may repeat the arguments in $\S \S \ref{subsec: modulus of sing}, \ref{subsec completion of perturbation lemma}$ with $\bar\delta$ in place of ${\delta}$. In this way, equations \eqref{xx}\eqref{modulus of a = sing}\eqref{matrix norm close} become, respectively,
\begin{eqnarray*}
&&\Big\| u - \T \Big(\frac{x-a}{|x-a|}\Big)\Big\|_{C^{0,\alpha}} \leq c_{27} \delta^{{1}/{2q}},\\
&&|a| \leq c_{27} \delta^{{1}/{q}},\\
&&\|\T - \id_{\R^3}\| \leq c_{27} \delta^{{1}/{2q}},
\end{eqnarray*}
where $c_{27}=c(q,\|\psi\|_{\rm Lip})$. Therefore, a straightforward adaptation of the proof in $\S \ref{subsec: 2<p}$ gives us
\begin{eqnarray*}
&&\|\eta^{-1}_i - \id_\Omega\|_{\rm Lip}+\|\eta_i - \id_\Omega\|_{\rm Lip} \leq c_{28} \delta_i^{1/2q},\\
&&\|u_i - v \circ \eta_i\|_{C^{0,\alpha'}} \leq c_{28} \delta_i^{1/2q}
\end{eqnarray*}
with $c_{28}=c(q,\|\psi\|_{\rm Lip})$.

We fix an arbitrary $q \in ]2,\infty[$ to conclude the proof of Theorem \ref{thm: W1,p, p>2} for $p=2$.

\section{Remarks and Prospective Questions}

{\bf 1.} It is interesting to investigate the boundary stability of minimising harmonic maps with axial symmetry ({\it cf.} Hardt--Lin--Poon \cite{hlp}, Hardt--Kinderlehrer--Lin \cite{hkl2}, and Hardt--Li \cite{us}). That is, the map $u: \ball \map \stwo$ is determined by its value on the ``orbit space'' $\{(r,z): 0\leq r\leq 1, r^2+z^2 \leq 1\}$, where $r=\sqrt{x^2+y^2}$ for $(x,y,z) \in \ball$. In this case, the singularities are at most a discrete set on the $z$-axis, but the proof of the inequality \eqref{caccioppoli} does not hold any more. The obstruction appears in an application of Fubini's theorem for the projection $\Pi_a$. 

{\bf 2.} We proved the stability of energy-minimisers under $W^{1,p}$-perturbations of the boundary maps under suitable uniqueness conditions, $p\geq 2$; Hardt--Lin \cite{hl} proved for $p=\infty$. This is in sharp contrast to the $p<2$ case in \cite{ms} by Mazowiecka--Strzelecki; also see Almgren--Lieb \cite{al}. In the nice recent work \cite{mms}, Mazowiecka--Mi\'{s}kiewicz--Schikorra proved (Theorem $7.1$ therein): 

Let $\Omega \subset \R^3$ be a bounded smooth domain. Let $s \in ]1/2,1]$ and $p\in [2,\infty[$. There are constants $R,\gamma$ depending only on $\Omega$ such that the following holds. Assume $v\in W^{1,2}(\Omega;\stwo)$ is the unique minimising harmonic map with $v|\p\Omega=\psi$. For any $\e>0$, there is $\delta=\delta(\Omega,\e,\psi)>0$ such that if $u \in W^{1,2}(\Omega,\stwo)$ is a minimising harmonic map with $u|\p\Omega=\varphi$ satisfying
\begin{equation*}
\sup_{\ball(y,\rho)\Subset\Omega, \, \rho < R} \Big\{ \rho^{sp-2} [\psi]^p_{W^{s,p}(\p\Omega\cap\ball(y,\rho))} \Big\}< \gamma
\end{equation*}
and
\begin{equation*}
[\psi-\varphi]_{W^{s,p}(\p\Omega)} \leq \delta,
\end{equation*}
then $u$ has the same number of singularities as $v$. Moreover, we have $\|u-v\|_{W^{1,2}}\leq \e$.

The above result in \cite{mms} by Mazowiecka--Mi\'{s}kiewicz--Schikorra has weaker regularity assumption on the boundary map --- $\psi \in W^{s,p}(\p\Omega,\stwo)$ --- compared to $\psi \in C^{1,\alpha}(\p\Omega,\stwo)$ in Theorem \ref{thm: W1,p, p>2} above. On the other hand, we bound the distance between $u$ and $v$ in a H\"{o}lder norm modulo bi-Lipschitz homeomorphisms, in comparison with the $W^{1,2}$-norm in \cite{mms}.

\end{document}